\documentclass{article}
\usepackage{amsmath,amsthm,amssymb,mathrsfs,amstext, titlesec,enumitem}
\makeatletter

\titleformat{\section}
{\normalfont\Large\bfseries}{\thesection}{1em}{}
\titleformat*{\subsection}{\Large\bfseries}
\titleformat*{\subsubsection}{\large\bfseries}
\titleformat*{\paragraph}{\large\bfseries}
\titleformat*{\subparagraph}{\large\bfseries}

\renewcommand{\@seccntformat}[1]{\csname the#1\endcsname.}

\renewenvironment{abstract}{%
    \if@twocolumn
      \section*{\abstractname}%
    \else 
      \begin{center}%
        {\bfseries \Large\abstractname\vspace{\z@}}
      \end{center}%
      \quotation
    \fi}
    {\if@twocolumn\else\endquotation\fi}

\makeatother
\theoremstyle{plain}
\newtheorem{thm}{Theorem}[section]
\newtheorem{lem}[thm]{Lemma}
\theoremstyle{definition}
\newtheorem{defn}[thm]{Definition}

\newtheorem{cor}[thm]{Corolarry}
\providecommand{\keywords}[1]{{\bf{Keywords:}} #1}
\providecommand{\subjectclass}[2]{\textbf{Mathematics subject classification 2020:} #1}

\title{{\bf Elementary characterization of essential $\mathcal{F}$-sets
and its combinatorial consequences}}

\author {Dibyendu De  \\  {\bf dibyendude@gmail.com}
  \footnote{ Department of Mathematics, 
          University of Kalyani, 
          Kalyani-741235,
          Nadia, West Bengal, India.
          }

           \and          
Pintu Debnath \\  {\bf pintumath1989@gmail.com}
\footnote{Department of Mathematics, Basirhat College, Basirhat -743412, North
24th parganas, West Bengal, India.
          }
           \and
Sayan Goswami \\ {\bf sayan92m@gmail.com} 
\footnote{Department of Mathematics, 
          University of Kalyani, 
          Kalyani-741235,
          Nadia, West Bengal, India
          }
          \footnote{The author acknowledges the grant of UGC NET SRF fellowship.}
}

\begin{document}
\maketitle

\begin{abstract}
\noindent There is a long history of studying Ramsey theory using
the algebraic structure of the Stone-\v{C}ech compactification of discrete
semigroup. It has been shown that various Ramsey theoretic structures
are contained in different algebraic large sets. In this article we
will deduce the combinatorial characterization of certain sets, that
are the member of the idempotent ultrafilters of the closed subsemigroup
of $\beta S$, arising from certain Ramsey family. In a special case 
when $S=\mathbb{N}$, we will deduce
that sets which are the members of all idempotent ultrafilters of
those semigroups contain certain additive-multiplicative structures. Later 
we will generalize this result for weak rings, where we will show a non-commutative
 version of the additive-multiplicative structure.
\end{abstract}
\subjectclass{05D10}\\

\noindent \keywords{Stone-\v{C}ech compactification, closed subsets of $\beta\mathbb{N}$,
Ramsay families}

\section{Introduction}

For a set $S$, let $\mathcal{P}\left(S\right)$ be the collection
of all subsets of $S$ and $\mathcal{P}_f\left(S\right)$ be the set
of all finite subsets of $S$. A collection $\mathcal{F\subseteq P}\left(S\right)\setminus\left\{ \emptyset\right\} $
is called upward hereditary if whenever $A\in\mathcal{F}$ and $A\subseteq B\subseteq S,$
then it follows that $B\in\mathcal{F}$. A non-empty and upward hereditary
collection $\mathcal{F\subseteq P}\left(S\right)\setminus\left\{ \emptyset\right\} $
is called a family. If $\mathcal{F}$ is a family, the dual family
$\mathcal{F}^{*}$ is given by,
\[
\mathcal{F}^{*}=\{E\subseteq S:\forall A\in\mathcal{F},E\cap A\neq\emptyset.\}
\]
A family $\mathcal{F}$ possesses the Ramsey property if whenever
$A\in\mathcal{F}$ and $A=A_{1}\cup A_{2}$ there is some $i\in\left\{ 1,2\right\} $
such that $A_{i}\in\mathcal{F}$. 

\noindent Throughout the article we will need the following definitions.
\begin{defn}
\cite{key-12} Let $\left(S,\cdot\right)$ be a discrete semigroup.
\end{defn}

\begin{enumerate}
\item The set $A$ is thick if and only if for any finite subset $F$ of
$S$, there exists an element $x\in S$ such that $F\cdot x\subset A$.
This means the sets which contains a translation of any finite subset.
For example, one can see $\cup_{n\in\mathbb{N}}\left[2^{n},2^{n}+n\right]$
is a thick set in $\mathbb{N}$.
\item The set $A$ is syndetic if and only if there exists a finite subset
$G$ of $S$ such that $\bigcup_{t\in G}t^{-1}A=S$. That is, with
a finite translation if, the set which covers the entire semigroup,
then it will be called a Syndetic set. For example, the set of even
and odd numbers are both syndetic in $\mathbb{N}$.
\item A set $A\subseteq S$ is $IP$ set if and only if there exists a sequence
$\left\{ x_{n}\right\} _{n=1}^{\infty}$ in $S$ such that $FP\left(\left\{ x_{n}\right\} _{n=1}^{\infty}\right)\subseteq A$.
Where 
\[
FP\left(\left\{ x_{n}\right\} _{n=1}^{\infty}\right)=\left\{ \prod_{n\in F}x_{n}:F\in\mathcal{P}_{f}\left(\mathbb{N}\right)\right\} 
\]
and $\prod_{n\in F}x_{n}$ to be the product in increasing order.
\item The sets which can be written as an intersection of a syndetic and
a thick set are called $\mathit{Piecewise}$ $\mathit{syndetic}$
sets. More formally a set $A$ is $\mathit{Piecewise}$ $\mathit{syndetic}$
if and only if there exists $G\in\mathcal{P}_{f}\left(S\right)$ such
that for every $F\in\mathcal{P}_{f}\left(S\right)$, there exists
$x\in S$ such that $F\cdot x\subseteq\bigcup_{t\in G}t^{-1}A$. Clearly
the thick sets and syndetic sets are natural examples of $\mathit{Piecewise}$
$\mathit{syndetic}$ sets. From definition one can immediately see
that $2\mathbb{N}\cap\bigcup_{n\in\mathbb{N}}\left[2^{n},2^{n}+n\right]$
is a non-trivial example of $\mathit{Piecewise}$ $\mathit{syndetic}$
sets in $\mathbb{N}$.
\item $\mathcal{T}=\,^{\mathbb{N}}S$.
\item For $m\in\mathbb{N}$, $\mathcal{J}_{m}=\left\{ \left(t\left(1\right),\ldots,t\left(m\right)\right)\in\mathbb{N}^{m}:t\left(1\right)<\ldots<t\left(m\right)\right\} .$
\item Given $m\in\mathbb{N}$, $a\in S^{m+1}$, $t\in\mathcal{J}_{m}$ and
$f\in F$, 
\[
x\left(m,a,t,f\right)=\left(\prod_{j=1}^{m}\left(a\left(j\right)\cdot f\left(t\left(j\right)\right)\right)\right)\cdot a\left(m+1\right)
\]
where the terms in the product $\prod$ are arranged in increasing
order.
\item $A\subseteq S$ is called a $J$-set iff for each $F\in\mathcal{P}_{f}\left(\mathcal{T}\right)$,
there exists $m\in\mathbb{N}$, $a\in S^{m+1}$, $t\in\mathcal{J}_{m}$
such that, for each $f\in\mathcal{T}$,
\[
x\left(m,a,t,f\right)\in A.
\]
\item If the semigroup $S$ is commutative, the definition is rather simple.
In that case, a set $A\subseteq S$ is a $J$-set if and only if whenever
$F\in\mathcal{P}_{f}\left(^{\mathbb{N}}S\right)$, there exist $a\in S$
and $H\in\mathcal{P}_{f}\left(\mathbb{N}\right)$, such that for each
$f\in F$, $a+\sum_{t\in H}f(t)\in A$.
\end{enumerate}
Let us talk about positive density sets of $\mathbb{N}$. Unmodified
positive density means positive asymptotic density, and set of subsets
of $\mathbb{N}$ with positive asymptotic density is not a family
since it is not closed under passage to supersets. (The asymptotic
density of $A\subseteq\mathbb{N}$ is $d\left(A\right)=lim_{n\rightarrow\infty}\frac{\mid A\cap\left\{ 1,2,\ldots,n\right\} \mid}{n}$
provided that limit exists and undefined otherwise.). But $\mathcal{F}$
is the family of subsets of $\mathbb{N}$ with positive upper asymptotic
density. The upper asymptotic density of $A\subseteq\mathbb{N}$ is
$d\left(A\right)=limsup_{n\rightarrow\infty}\frac{\mid A\cap\left\{ 1,2,\ldots,n\right\} \mid}{n}$.

There are many families $\mathcal{F}$ with Ramsay property, where
for uniformity we consider following families for $\mathbb{\mathbb{N}}$.
\begin{itemize}
\item The infinite sets,
\item The piecewise syndetic sets,
\item The sets of positive upper asymptotic density,
\item The set containing arbitrary large arithmetic progression,
\item The set with property that $\sum_{n\in A}\frac{1}{n}=\infty$,
\item The $J$-sets,
\item The $IP$-sets.
\end{itemize}
\vspace{0.1in}

\subsubsection{A brief review of Topological algebra: \\
}

Let us recall some basic algebraic structure of the Stone-\v{C}ech
compactification. The set $\{\overline{A}:A\subset S\}$ is a basis
for the closed sets of $\beta S$. The operation `$\cdot$' on $S$
can be extended to the Stone-\v{C}ech compactification $\beta S$
of $S$ so that$(\beta S,\cdot)$ is a compact right topological semigroup
(meaning that for any \  is continuous) with $S$ contained in its
topological center (meaning that for any $x\in S$, the function $\lambda_{x}:\beta S\rightarrow\beta S$
defined by $\lambda_{x}(q)=x\cdot q$ is continuous). This is a famous
Theorem due to Ellis that if $S$ is a compact right topological semigroup
then the set of idempotents $E\left(S\right)\neq\emptyset$. A non-empty
subset $I$ of a semigroup $T$ is called a $\textit{left ideal}$
of $S$ if $TI\subset I$, a $\textit{right ideal}$ if $IT\subset I$,
and a $\textit{two sided ideal}$ (or simply an $\textit{ideal}$)
if it is both a left and right ideal. A $\textit{minimal left ideal}$
is the left ideal that does not contain any proper left ideal. Similarly,
we can define $\textit{minimal right ideal}$ and $\textit{smallest ideal}$.

Any compact Hausdorff right topological semigroup $T$ has the smallest
two sided ideal

\[
\begin{array}{ccc}
K(T) & = & \bigcup\{L:L\text{ is a minimal left ideal of }T\}\\
 & = & \,\,\,\,\,\bigcup\{R:R\text{ is a minimal right ideal of }T\}.
\end{array}
\]

Given a minimal left ideal $L$ and a minimal right ideal $R$, $L\cap R$
is a group, and in particular contains an idempotent. If $p$ and
$q$ are idempotents in $T$ we write $p\leq q$ if and only if $pq=qp=p$.
An idempotent is minimal with respect to this relation if and only
if it is a member of the smallest ideal $K(T)$ of $T$. Given $p,q\in\beta S$
and $A\subseteq S$, $A\in p\cdot q$ if and only if the set $\{x\in S:x^{-1}A\in q\}\in p$,
where $x^{-1}A=\{y\in S:x\cdot y\in A\}$. See \cite{key-12} for an
elementary introduction to the algebra of $\beta S$ and for any unfamiliar
details.

It will be easy to check that the family $\mathcal{F}$ has the Ramsey
property if and only if the family $\mathcal{F}^{*}$ is a filter. For a family
$\mathcal{F}$ with the Ramsey property, let $\beta(\mathcal{F})=\{p\in\beta S:p\subseteq\mathcal{F}\}$.
Then the following from \cite[Theorem 5.1.1]{key-4}.
\begin{thm}
Let $S$ be a discrete set. For every family $\mathcal{F\subseteq P}\left(S\right)$
with the Ramsay property, $\beta\left(\mathcal{F}\right)\subseteq\beta S$
is closed. Furthermore, $\mathcal{F}=\cup\beta\left(\mathcal{F}\right)$.
Also if $K\subseteq\beta S$ is closed, $\mathcal{F}_{K}=\left\{ E\subseteq S:\overline{E}\cap K\neq\emptyset\right\} $
is a family with the Ramsay property and $\overline{K}=\beta\left(\mathcal{F}_{K}\right)$.
\end{thm}

Let $S$ be a discrete semigroup, then for every family $\mathcal{F\subseteq P}\left(S\right)$
with the Ramsay property, $\beta\left(\mathcal{F}\right)\subseteq\beta S$
is closed. If $\beta\left(\mathcal{F}\right)$ be a subsemigroup of
$\beta S$, then $E\left(\beta\mathcal{F}\right)\neq\emptyset$. But
may not be subsemigroup. For example, let $\mathcal{F}=\mathcal{IP}$,
the family of IP- sets. It is easy to show that $\beta\left(\mathcal{F}\right)=\beta\left(\mathcal{IP}\right)=E\left(\beta S\right)$.
But $E\left(\beta S\right)$ is not a subsemigroup of $\beta S$.
\begin{defn}
Let $\mathcal{F}$ be a family with Ramsay property such that $\beta(\mathcal{F})$
is a subsemigroup of $\beta S$ and $p$ be an idempotent in $\beta(\mathcal{F})$,
then each member of $p$ is called essential $\mathcal{F}$-set. And
$A\subset S$ is called essential $\mathcal{F}^{\star}$-set if $A$
intersects with all essential $\mathcal{F}$-sets. That is a set is
essential $\mathcal{F}^{\star}$ if and only if $A\in p$, for every
$p\in E\left(\beta(\mathcal{F})\right).$
\end{defn}

The family $\mathcal{F}$ is called left (right) shift-invariant if
for all $s\in S$ and all $E\in\mathcal{F}$ one has $sE\in\mathcal{F}(Es\in\mathcal{F})$.
The family $\mathcal{F}$ is called left (right) inverse shift-invariant
if for all $s\in S$ and all $E\in\mathcal{F}$ one has $s^{-1}E\in\mathcal{F}(Es^{-1}\in\mathcal{F})$.
From\cite[Theorem 5.1.2]{key-4} we have the following one:
\begin{thm}
If $\mathcal{F}$ is a family having the Ramsey property then $\beta\mathcal{F}\subseteq\beta S$
is a left ideal if and only if $\mathcal{F}$ is left shift-invariant.
Similarly, $\beta\mathcal{F}\subseteq\beta S$ is a right ideal if
and only if $\mathcal{F}$ is right shift-invariant.
\end{thm}

From \cite[Theorem 5.1.10]{key-4}, we can identify those families
$\mathcal{F}$ with Ramsey property for which $\beta\left(\mathcal{F}\right)$
is a subsemigroup of $\beta S$ . The condition is a rather technical
weakening of left shift-invariance.
\begin{thm}
Let $S$ be any semigroup, and let $\mathcal{F}$ be a family of subsets
of $S$ having the Ramsey property. Then the following are equivalent:
\begin{enumerate}
\item $\beta\left(\mathcal{F}\right)$ is a subsemigroup of $\beta S$.

\item $\mathcal{F}$ has the following property: If $E\subseteq S$
is any set, and if there is $A\in\mathcal{F}$ such that for all finite
$H\subseteq A$ one has $\left(\cap_{q\in H}x^{-1}E\right)\in\mathcal{F}$,
then $E\in\mathcal{F}$.
\end{enumerate}
\end{thm}

Let us abbreviate the family of infinite sets as $\mathcal{IF}$,
the family of piecewise syndetic sets as $\mathcal{PS}$, the family
of positive upper asymptotic density as $\Delta$, the family sets
containing arithmetic progression of arbitrary length as $\mathcal{AP}$,
the family of sets with the property $\sum_{n\in A}\frac{1}{n}=\infty$
as $\mathcal{HSD}$ and the family of $J$-sets as $\mathcal{J}$.

From the above definition together with the abbreviations, we get
quasi central set is an essential $\mathcal{PS}$-set, $D$-set is
an essential $\Delta$-set and $C$-set is an essential $\mathcal{J}$-set.

\section{Elementary characterization of essential $\mathcal{F}$-sets}

In \cite[Theorem 5.2.3]{key-4}, the author has established dynamical
characterization of essential $\mathcal{F}$-sets. But the elementary
characterization of essential $\mathcal{F}$-sets are still unknown.
Although elementary characterization of quasi central-sets and $C$-sets
are known from \cite[Theorem 3.7]{key-10} and \cite[Theorem 2.7]{key-11}
respectively. Since quasi central sets and $C$ sets comes from the
settings of essential $\mathcal{F}$-set and this fact confines the
fact that essential $\mathcal{F}$-sets might have elementary characterization.
In this section we will prove the supposition that elementary characterization
of essential $\mathcal{F}$-sets could be found exactly the same way
what the authors did in \cite{key-11} for $C$-sets.

Let $\omega$ be the first infinite ordinal and each ordinal indicates
the set of all it's predecessor. In particular, $0=\emptyset,$ for
each $n\in\mathbb{N},\:n=\left\{ 0,1,...,n-1\right\} $.
\begin{defn}
Let us recall the following definitions from \cite[Definition 2.5]{key-11}.
\begin{enumerate}
\item If $f$ is a function and $dom\left(f\right)=n\in\omega$, then for
all $x$, $f^{\frown}x=f\cup\left\{ \left(n,x\right)\right\} $.
\item Let $T$ be a set functions whose domains are members of $\omega$.
For each $f\in T$, $B_{f}\left(T\right)=\left\{ x:f^{\frown}x\in T\right\} .$
\end{enumerate}
We recall the following lemma from \cite[Lemma 2.6]{key-11} which
is key to our characterization of essential $\mathcal{F}$-sets.
\end{defn}

\begin{lem}
\label{F lemma}Let $p\in\beta S$. Then $p$ is an idempotent if
and only if for each $A\in p$ there is a non-empty set $T$ of functions
such that
\begin{enumerate}
\item For all $f\in T$, $dom\left(f\right)\in\omega$ and $range\left(f\right)\subseteq A$.
\item For all $f\in T$, $B_{f}\left(T\right)\in p$.
\item For all $f\in T$ and any $x\in B_{f}\left(T\right)$, $B_{f^{\frown}x}\left(T\right)\subseteq x^{-1}B_{f}\left(T\right)$.
\end{enumerate}
\end{lem}

The following theorem is the characterization of essential $\mathcal{F}$-
set.
\begin{thm}
\label{thm}Let $\left(S,\cdot\right)$ be a semigroup, and assume
that $\mathcal{F}$ is a family of subsets of $S$ with the Ramsay
property such that $\beta\left(\mathcal{F}\right)$ is a subsemigroup
of $\beta S$. Let $A\subseteq S$. Statements $(1)$, $(2)$ and
$(3)$ are equivalent and are implied by statement $(4)$. If $S$
is countable, then all the five statements are equivalent.
\begin{enumerate}
\item $A$ is an essential $\mathcal{F}$-set.
\item There is a non empty set $T$ of functions such that:
\begin{enumerate}
\item For all $f\in T$, $\text{domain}\left(f\right)\in\omega$ and $rang\left(f\right)\subseteq A$.
\item For all $f\in T$ and all $x\in B_{f}\left(T\right)$, $B_{f^{\frown}x}\subseteq x^{-1}B_{f}\left(T\right)$.
\item For all $F\in\mathcal{P}_{f}\left(T\right)$, $\cap_{f\in F}B_{f}(T)$
is a $\mathcal{F}$-set.
\end{enumerate}
\item \label{3chain} There is a downward directed family $\left\langle C_{F}\right\rangle _{F\in I}$
of subsets of $A$ such that:
\begin{enumerate}
\item For each $F\in I$ and each $x\in C_{F}$ there exists $G\in I$ with
$C_{G}\subseteq x^{-1}C_{F}$.
\item For each $\mathcal{F}\in\mathcal{P}_{f}\left(I\right),\,\bigcap_{F\in\mathcal{F}}C_{F}$
is a $\mathcal{F}$-set.
\end{enumerate}
\item \label{chain} There is a decreasing sequence $\left\langle C_{n}\right\rangle _{n=1}^{\infty}$
of subsets of $A$ such that
\begin{enumerate}
\item For each $n\in\mathbb{N}$ and each $x\in C_{n}$, there exists $m\in\mathbb{N}$
with $C_{m}\subseteq x^{-1}C_{n}$.
\item For each $n\in\mathbb{N}$, $C_{n}$ is a $\mathcal{F}$-set.
\end{enumerate}
\end{enumerate}
\end{thm}

\begin{proof}
$(1)$ $\Rightarrow$ $(2)$ As $A$ be an essential $\mathcal{F}$
set, then there exists an idempotent $p\in\beta\left(\mathcal{F}\right)$
such that $A\in p$. Pick a set $T$ of functions as guaranteed by\textbf{
}Lemma \ref{F lemma}. Conclusions $(a)$ and $(b)$ hold directly.
Given $F\in\mathcal{P}_{f}\left(T\right)$, $B_{f}\in p$ for all
$f\in F$, hence $\bigcap_{f\in F}B_{f}\in p$ and so $\bigcap_{f\in F}B_{f}$
is a $\mathcal{F}$-set.\vspace{0.1in}

$(2)\Rightarrow(3)$ Let $T$ be guaranteed by $(2)$. Let $I=\mathcal{P}_{f}\left(T\right)$
and for each $F\in I$, let $C_{F}=\bigcap_{f\in F}B_{f}$. Then directly
each $C_{F}$ is a $\mathcal{F}$-set. Given $\mathcal{F}\in\mathcal{P}_{f}\left(I\right)$,
if $G=\bigcup\mathcal{F}$, then $\bigcap_{F\in\mathcal{F}}C_{F}=C_{G}$
and is therefore a $\mathcal{F}$-set. To verify $(a)$, let $F\in I$
and let $x\in C_{F}$. Let $G=\left\{ f^{\frown}x:f\in F\right\} $.
For each $f\in F$, $B_{f^{\frown}x}\subseteq x^{-1}B_{f}$ and so
$C_{G}\subseteq x^{-1}C_{F}$.\vspace{0.1in}

$(3)\Rightarrow(1)$ Let $\langle C_{F}\rangle$ is guaranteed by
(3). Let $M=\bigcap_{F\in I}\overline{C_{F}}$. By\cite[Theorem 4.20]{key-12},
$M$ is a subsemigroup of $\beta S$. By\cite[Theorem 3.11]{key-12}
there is some $p\in\beta S$ such that $\left\{ C_{F}:F\in I\right\} \subseteq p\subseteq\mathcal{F}$.
Therefore $M\cap\beta\left(\mathcal{F}\right)\neq\emptyset$; and
so $M\cap\beta\left(\mathcal{F}\right)$ is a compact subsemigroup
of $\beta S$. Thus there is an idempotent $p\in M\cap\beta\left(\mathcal{F}\right)$,
and so  $A$ is an essential $\mathcal{F}$-set. \vspace{0.1in}

It is trivial that $(4)\Rightarrow(3)$. Assume now that $S$ is countable.
We shall show that $(2)\Rightarrow(4)$. So let $T$ be as guaranteed
by $(2)$. Then $T$ is countable so enumerate $T$ as $\left\{ f_{n}:n\in\mathbb{N}\right\} $.
For $n\in\mathbb{N}$, let $C_{n}=\bigcap_{k=1}^{n}B_{f_{k}}$. Then
each $C_{n}$ is a $\mathcal{F}$-set. Let $n\in\mathbb{N}$ and let
$x\in C_{n}$. Pick $m\in\mathbb{N}$ such that 
\[
\left\{ f_{k}^{\frown}x:k\in\left\{ 1,2,\ldots,n\right\} \right\} \subseteq\left\{ f_{1},f_{2},\ldots,f_{m}\right\} .
\]
 Then $C_{m}\subseteq x^{-1}C_{n}$.
\end{proof}
\noindent Now, we will conclude this section with the following fascinating
results:
\begin{cor}
Let $\left(S,\cdot\right)$ be a countable discrete semigroup and
let $\mathcal{F}$ be a left inverse shift invariant family with Ramsay
property such that $\beta\left(\mathcal{F}\right)$ is subsemigroup
of $\beta S$. If there exists a sequence $\langle x_{n}\rangle_{n=1}^{\infty}$
in $S$ such that $FP\left(\langle x_{n}\rangle_{n=1}^{\infty}\right)$
is $\mathcal{F}$-set, then $FP\left(\langle x_{n}\rangle_{n=1}^{\infty}\right)$
is essential $\mathcal{F}$-set.
\end{cor}

\begin{proof}
Now choose arbitrarily $m\in\mathbb{N}$. Then
\[
FP\left(\langle x_{n}\rangle_{n=1}^{\infty}\right)=FP\left(\langle x_{n}\rangle_{n=m}^{\infty}\right)\cup FP\left(\langle x_{n}\rangle_{n=1}^{m-1}\right)
\]
\[
\qquad\qquad\qquad\qquad\qquad\;\;\;\cup\left\{ tFP\left(\langle x_{n}\rangle_{n=m}^{\infty}\right):t\in FP\left(\langle x_{n}\rangle_{n=1}^{m-1}\right)\right\} 
\]
\noindent As $\mathcal{F}$ is a Ramsay family, $FP\left(\langle x_{n}\rangle_{n=1}^{\infty}\right)$
is $\mathcal{F}$-set, and $FP\left(\langle x_{n}\rangle_{n=1}^{m-1}\right)$
is a finite set, we have either $FP\left(\langle x_{n}\rangle_{n=m}^{\infty}\right)\in\mathcal{F}$
or $tFP\left(\langle x_{n}\rangle_{n=m}^{\infty}\right)\in\mathcal{F}$
for some $t\in FP\left(\langle x_{n}\rangle_{n=1}^{m-1}\right)$.
Now $\mathcal{F}$ being left inverse shift invariant, in either cases
$FP\left(\langle x_{n}\rangle_{n=m}^{\infty}\right)$ is $\mathcal{F}$-
set. Let us consider the sequence 
\[
FP\left(\langle x_{n}\rangle_{n=1}^{\infty}\right)\supseteq FP\left(\langle x_{n}\rangle_{n=2}^{\infty}\right)\supseteq\cdots\supseteq FP\left(\langle x_{n}\rangle_{n=k}^{\infty}\right)\supseteq\cdots.
\]
 Let $C_{i}=FP\left(\langle x_{n}\rangle_{n=i+1}^{\infty}\right)$
for all $i\in\mathbb{N}$. Then for any $n\in\mathbb{N}$ , any $x\in C_{n}$
, we have $l\in\mathbb{N}$ such that $C_l\subseteq x^{-1}C_{n}$.
And this concludes that $FP\left(\left\{ x_{n}\right\} _{n=1}^{\infty}\right)$
is an essential $\mathcal{F}$-set from the previous theorem.
\end{proof}
\begin{cor}
If $S$ is a group and $G$ be a subgroup of $S$ which is a $\mathcal{F}$-
set, then it is an essential $\mathcal{F}$- set.
\end{cor}

\begin{proof}
As
\[
G\supseteq G\supseteq\cdots\supseteq G\supseteq\cdots
\]
is the necessary chain of condition \ref{3chain} of theorem \ref{thm},
it follows.
\end{proof}

\section{Combined additive and multiplicative structure}

Given a sequence $\langle x_{n}\rangle_{n=1}^{\infty}$ in $\mathbb{N}$,
we say that $\langle y_{n}\rangle_{n=1}^{\infty}$ is a sum subsystem
of $\langle x_{n}\rangle_{n=1}^{\infty}$ provided there exists a
sequence $\langle H_{n}\rangle_{n=1}^{\infty}$ of non-empty finite
subset such that $\text{max}H_{n}<\text{min}H_{n+1}$ and $y_{n}=\sum_{t\in H_{n}}x_{t}$
for each $n\in\mathbb{N}$. In \cite{key-3} N. Hindman and V. Bergelson
proved the following theorem.
\begin{thm}
Let $\langle x_{n}\rangle_{n=1}^{\infty}$ be a sequence in $\mathbb{N}$
and $A$ be IP$^{\star}$-set in $\left(\mathbb{N},+\right)$. Then
there exists a subsystem $\langle y_{n}\rangle_{n=1}^{\infty}$ of
$\langle x_{n}\rangle_{n=1}^{\infty}$ such that $FS\left(\langle y_{n}\rangle_{n=1}^{\infty}\right)\cup FP\left(\langle y_{n}\rangle_{n=1}^{\infty}\right)\subseteq A$.
\end{thm}

In \cite[Theorem 2.4]{key-6}, it was proved that central$^{\star}$
sets also possess some IP$^{\star}$-set-like properties for some
specified sequences called minimal sequence\cite[Definition 2.4]{key-6}:
\begin{defn}
A sequence $\langle x_{n}\rangle_{n=1}^{\infty}$ in $\mathbb{N}$
is minimal sequence if
\[
\cap_{m=1}^{\infty}cl\left(FS\left(\langle x_{n}\rangle_{n=m}^{\infty}\right)\right)\cap K\left(\beta\mathbb{N}\right)\neq\emptyset.
\]
\end{defn}

It is known that $\langle2^{n}\rangle_{n=1}^{\infty}$ is a minimal
sequence while the sequence $\langle2^{2n}\rangle_{n=1}^{\infty}$
is not a minimal sequence. And in\cite[Theorem 2.4]{key-6}, it was
proved the following substantial multiplicative result of central$^{\star}$
sets.
\begin{thm}
Let $\langle x_{n}\rangle_{n=1}^{\infty}$ be a minimal sequence in
$\mathbb{N}$ and $A$ be central$^{\star}$ set in $\left(\mathbb{N},+\right)$.
Then there exists a subsystem $\langle y_{n}\rangle_{n=1}^{\infty}$
of $\langle x_{n}\rangle_{n=1}^{\infty}$ such that $FS\left(\langle y_{n}\rangle_{n=1}^{\infty}\right)\cup FP\left(\langle y_{n}\rangle_{n=1}^{\infty}\right)\subseteq A$.
\end{thm}

In \cite[Theorem 2.10]{key-5}, it was established an analogue version
of the above theorem in case of $C^{\star}$ sets for some specific
type of sequences called almost minimal sequence \cite[Definition 2.3]{key-5}:
\begin{defn}
A sequence $\langle x_{n}\rangle_{n=1}^{\infty}$ in $\mathbb{N}$
is almost minimal sequence if $$\cap_{m=1}^{\infty}cl\left(FS\left(\langle x_{n}\rangle_{n=m}^{\infty}\right)\right)\cap J\left(\mathbb{N}\right)\neq\emptyset$$.
\end{defn}

\noindent In \cite[Theorem 2.7]{key-5}, it has been characterized the almost
minimal sequences by the following theorem.
\begin{thm}
In $(\mathbb{N},+)$ the following conditions are equivalent:
\begin{enumerate}
\item $\langle x_{n}\rangle_{n=1}^{\infty}$ is almost minimal sequence.
\item $FS(\langle x_{n}\rangle_{n=1}^{\infty})$ is a $J$-set.
\item There is an idempotent in $\cap_{m=1}^{\infty}cl\left(FS\left(\langle x_{n}\rangle_{n=m}^{\infty}\right)\right)\cap J\left(\mathbb{N}\right)$.
\end{enumerate}
\end{thm}

\noindent Now we are in position  to state the main theorem of \cite[Theorem 2.10]{key-5}:
\begin{thm}
Let $\langle x_{n}\rangle_{n=1}^{\infty}$ be a minimal sequence in
$\mathbb{N}$ and $A$ be $C^{\star}$ set in $\left(\mathbb{N},+\right)$.
Then there exists a subsystem $\langle y_{n}\rangle_{n=1}^{\infty}$
of $\langle x_{n}\rangle_{n=1}^{\infty}$ such that $$FS\left(\langle y_{n}\rangle_{n=1}^{\infty}\right)\cup FP\left(\langle y_{n}\rangle_{n=1}^{\infty}\right)\subseteq A$$.
\end{thm}

As we know that $C$-sets are essential $\mathcal{J}$-sets, the above
theorem motives us to think some analogue result for essential $\mathcal{F}$-sets.
First let us define , $\mathcal{F}$-minimal sequence.
\begin{defn}
A sequence $\langle x_{n}\rangle_{n=1}^{\infty}$ in $\mathbb{N}$
is $\mathcal{F}$-minimal sequence if 
\[
\cap_{m=1}^{\infty}cl\left(FS\left(\langle x_{n}\rangle_{n=m}^{\infty}\right)\right)\cap\beta\left(F\right)\neq\emptyset.
\]
\end{defn}

We can characterize $\mathcal{F}$- minimal sequences as like as almost
minimal sequence given below and can be proved in the same way as
the author did in \cite[Theorem 2.7]{key-5} for almost minimal sequences:
\begin{thm} \label{equiv}
For an inverse shift invariant family $\mathcal{F}$ in $(\mathbb{N},+)$
with Ramsay property such that $\beta\left(\mathcal{F}\right)$ is a subsemigroup
of $\beta S$, the following conditions are equivalent:
\begin{enumerate}
\item $\langle x_{n}\rangle_{n=1}^{\infty}$ is almost $\mathcal{F}$- minimal
sequence.
\item $FS(\langle x_{n}\rangle_{n=1}^{\infty})\in q$, for some $q\in\beta(\mathcal{F})$.
\item There is an idempotent in $\cap_{m=1}^{\infty}cl(FS(\langle x_{n}\rangle_{n=m}^{\infty}))\cap\beta(\mathcal{F})$.
\end{enumerate}
\end{thm}

\begin{proof}
$(1)\implies(2)$ follows from definition.\vspace{0.1in}

$(2)\implies(3)$ Since $FS\left(\langle x_{n}\rangle_{n=1}^{\infty}\right)\in q\in\beta(\mathcal{F})$
we get $cl\left(FS\left(\langle x_{n}\rangle_{n=1}^{\infty}\right)\right)\cap\beta\left(\mathcal{F}\right)\neq\emptyset$.
From \cite[Lemma 5.11]{key-12}, choose $\cap_{m=1}^{\infty}cl\left(FS\left(\langle x_{n}\rangle_{n=m}^{\infty}\right)\right)$.
It will easy to see that $\cap_{m=1}^{\infty}cl\left(FS\left(\langle x_{n}\rangle_{n=1}^{\infty}\right)\right)$
is a closed subsemigroup of $\beta\mathbb{N}$ and as well as $\beta(\mathcal{F})$
is also closed subsemigroup $\beta\mathbb{N}$. Hence $\,\cap_{m=1}^{\infty}cl\left(FS\left(\langle x_{n}\rangle_{n=m}^{\infty}\right)\right)\cap\beta\left(\mathcal{F}\right)$
is a compact subsemigroup of $\left(\beta\mathbb{N},+\right)$. So
it will be sufficient to check that $\cap_{m=1}^{\infty}cl\left(FS\left(\langle x_{n}\rangle_{n=m}^{\infty}\right)\right)\cap\beta\left(\mathcal{F}\right)\neq\emptyset$.

Now choose arbitrarily $m\in\mathbb{N}$ and then $FS\left(\langle x_{n}\rangle_{n=1}^{\infty}\right)=FS\left(\langle x_{n}\rangle_{n=m}^{\infty}\right)\cup FS\left(\langle x_{n}\rangle_{n=1}^{m-1}\right)\cup\left\{ t+FS\left(\langle x_{n}\rangle_{n=m}^{\infty}\right):t\in FS\left(\langle x_{n}\rangle_{n=1}^{m-1}\right)\right\} $
and so we have one of the followings:
\begin{enumerate}
\item $FS\left(\langle x_{n}\rangle_{n=m}^{\infty}\right)\in p$
\item $FS\left(\langle x_{n}\rangle_{n=1}^{m-1}\right)\in p$
\item $t+FS\left(\langle x_{n}\rangle_{n=m}^{\infty}\right)\in p$ for some
$t\in FS\left(\langle x_{n}\rangle_{n=1}^{m-1}\right)$.
\end{enumerate}
Now $(2)$ is not possible as in that case $p$ will be a member of
principle ultrafilter. If $(1)$ holds then we have done. Now if we
assume $(3)$ holds then for some $t\in FS\left(\langle x_{n}\rangle_{n=1}^{m-1}\right)$,
we have $t+FS\left(\langle x_{n}\rangle_{n=m}^{\infty}\right)\in p$.
Choose $q\in cl\left(FS\left(\langle x_{n}\rangle_{n=m}^{\infty}\right)\right)$
so that $t+q=p$. Now for every $F\in q$, $t\in\left\{ n\in\mathbb{N}:-n+\left(t+F\right)\in q\right\} $
so that $t+F\in p$. Since $\mathcal{F}$-sets are inverse shift invariant invariant,
$F$ is a $\mathcal{F}$-sets. We have $q\in\beta\left(\mathcal{F}\right)\cap cl\left(FS\left(\langle x_{n}\rangle_{n=m}^{\infty}\right)\right)$.\vspace{0.1in}

$(3)\implies(1)$ follows from definition of $\mathcal{F}$- minimal
sequence and condition $(3)$.
\end{proof}
To prove the main theorem, we need the following two lemmas are essential.
\begin{lem}
Let $\mathcal{F}$ be a dilation invariant family (i.e. the family is invariant under taking product by any element of $\mathbb{N}$) with Ramsay property such that $\beta\left(\mathcal{F}\right)$ is a subsemigroup
of $\beta S$.
If $A$ be an essential $\mathcal{F}$-set in $(\mathbb{N},+)$ then
$nA$ is also an essential $\mathcal{F}$-set in $(\mathbb{N},+)$
for any $n\in\mathbb{N}$.
\end{lem}

\begin{proof}
If $A$ be an essential $\mathcal{F}$-set, then by elementary characterization
of essential $\mathcal{F}$-set, we get a sequence of $\mathcal{F}$-sets
$\left\langle C_{k}\right\rangle _{k=1}^{\infty}$ with 
\[
A\supseteq C_{1}\supseteq C_{2}\supseteq\cdots
\]
satisfying property \ref{chain} of theorem \ref{thm}. Now consider
the sequence $\left\langle nC_{k}\right\rangle _{k=1}^{\infty}$ of
$\mathcal{F}$-sets which satisfies 
\[
nA\supseteq nC_{1}\supseteq nC_{2}\supseteq\cdots
\]
 and for each $k\in\mathbb{N}$ and each $t\in nC_{k}$, there exists
$p\in\mathbb{N}$ with $nC_{p}\subseteq-t+nC_{k}$. This proves that
$nA$ is an essential $\mathcal{F}$-set in $(\mathbb{N},+)$ for
any $n\in\mathbb{N}$.
\end{proof}
We get another lemma given below.
\begin{lem}
Let $\mathcal{F}$ be a dilation invariant family with Ramsay property such that $\beta\left(\mathcal{F}\right)$ is a subsemigroup
of $\beta S$.
If $A$ be an essential $\mathcal{F}^{\star}$-set in $(\mathbb{N},+)$
then $n^{-1}A$ is also a essential $\mathcal{F^{\star}}$-set in
$(\mathbb{N},+)$ for any $n\in\mathbb{N}$.
\end{lem}

\begin{proof}
It is sufficient to show that for any essential $\mathcal{F}$-set
$B$, $B\cap n^{-1}A\neq\emptyset$. Since $B$ is essential $\mathcal{F}$-set
, $nB$ is essential $\mathcal{F}$-set and $A\cap nB\neq\emptyset$.
Choose $m\in A\cap nB$ and $k\in B$ such that $m=nk$. Therefore
$k=m/n\in n^{-1}A$ so $B\cap n^{-1}A\neq\emptyset$.
\end{proof}
Now we will show that all $\mathcal{F}^{\star}$-set have a substantial
multiplicative property.
\begin{thm}
Let $\mathcal{F}$ be an inverse shift invariant and dilation invariant family with Ramsay property such that $\beta\left(\mathcal{F}\right)$ is a subsemigroup
of $\beta S$.
Let $\langle x_{n}\rangle_{n=1}^{\infty}$ be a $\mathcal{F}$-minimal
sequence and $A$ be a an essential $\mathcal{F}$-set in $(\mathbb{N},+)$. Then there exists a sum
subsystem $\langle y_{n}\rangle_{n=1}^{\infty}$ of $\langle x_{n}\rangle_{n=1}^{\infty}$
such that\textup{ $FS(\langle y_{n}\rangle_{n=1}^{\infty})\cup FP(\langle y_{n}\rangle_{n=1}^{\infty})\subseteq A$.}
\end{thm}

\begin{proof}
Since $\langle x_{n}\rangle_{n=1}^{\infty}$ is a $\mathcal{F}$-minimal
sequence in $\mathbb{N}$, we can find some essential idempotent $p\in E\left(\beta\left(\mathcal{\mathcal{F}}\right)\right)$
for which $FS(\langle x_{n}\rangle_{n=m}^{\infty})\in p$ for each $m\in \mathbb{N}$. Since A
be an essential $\mathcal{F}^{\star}$-set for every $n\in\mathbb{N}$, $n^{-1}A\in p$.
Let $A^{\star}=\{n\in A:-n+A\in p\}$, then $A^{\star}\in p$. We
can choose $y_{1}\in A^{\star}\cap FS(\langle x_{n}\rangle_{n=1}^{\infty})$.
Inductively, let $m\in\mathbb{N}$ and $\langle y_{i}\rangle_{i=1}^{m}$,
$\langle H_{i}\rangle_{i=1}^{m}$ in $\mathcal{P}_{f}(\mathbb{N})$
be chosen with the following property:
\begin{enumerate}
\item $i\in\{1,2,\ldots\,,m-1\},\,\text{max}H_{i}<\text{min}H_{i+1}.$
\item $\text{If}\,y_{i}={\displaystyle \sum_{t\in H_{i}}x_{t},\,\text{then}\,\sum_{t\in H_{i}}x_{t}\in A^{\star},\,\text{and}\,FP(\langle y_{i}\rangle_{i=1}^{m})\subseteq A^\star}.$
\end{enumerate}
We observe that $\{\,\sum_{t\in H}x_{t}:\,H\in\mathcal{P}_{f}(\mathbb{N}),\,\text{min}H>\text{max}H_{m}\}\in p$.
Let us set $B=\{\,\sum_{t\in H}x_{t}:\,H\in\mathcal{P}_{f}(\mathbb{N}),\,\text{min}H>\text{max}H_{m}\}$,
$E_{1}=FS(\langle y_{t}\rangle_{n=1}^{m})$ and $E_{2}=FP(\langle y_{t}\rangle_{n=1}^{m})$.
Now consider $D=B\cap A^{\star}\cap\bigcap_{s\in E_{1}}\left(-s+A^{\star}\right)\cap\bigcap_{s\in E_{2}}\left(s^{-1}A^{\star}\right)$.
Then $D\in p$. Choose $y_{m+1}\in D$ and $H_{m+1}\in\mathcal{P}_{f}\left(\mathbb{N}\right)$
 such that $\text{min}H_{m+1}>\text{max}H_{m}$. Putting $y_{m+1}=\sum_{t\in H_{m+1}}x_{t}$,
it shows that the induction can be continued and proves the theorem.
\end{proof}

\section{Essential $\mathcal{F}^{\star}$-sets in weak rings}

In this section we extend the previous theorem to a much wider class,
called \textquotedblleft weak rings\textquotedblright{} and start
with the following definition:
\begin{defn}
\cite [Definition 16.33, Page 419]{key-7}
\begin{enumerate}
\item A left weak ring is a triple $\left(S,+,\cdot\right)$ such that $\left(S,+\right)$
and $\left(S,.\right)$ are semigroups and the left distributive law
holds. That is, for all $x,y,z\in S$ one has$x\cdot\left(y+z\right)=x\cdot y+x\cdot z$.
\item A right weak ring is a triple $\left(S,+,\cdot\right)$ such that
$\left(S,+\right)$ and $\left(S,.\right)$ are semigroups and the
right distributive law holds. That is, for all $x,y,z\in S$ one has$\left(x+y\right)\cdot z=x\cdot z+y\cdot z$.
\item A weak ring is a triple $\left(S,+,\cdot\right)$ which is both a
left weak ring and a right weak ring.
\end{enumerate}
\end{defn}

\noindent Dilation invariance of a family is defined as:
\begin{defn}
This is similar as for dilation invariance of $\mathbb{N}$.
\begin{enumerate} 
\item Let $\left(S,+,\cdot\right)$ be a left weak ring. A family $\mathcal{F}$
is called left dilation invariant if for any $s\in S$ and $A\in\mathcal{F}$,
$sA\in\mathcal{F}$.
\item Let $\left(S,+,\cdot\right)$ be a right weak ring. A family $\mathcal{F}$
is called right dilation invariant if for any $s\in S$ and $A\in\mathcal{F}$,
$As\in\mathcal{F}$.
\item Let $\left(S,+,\cdot\right)$ be a weak ring. A family $\mathcal{F}$
is called dilation invariant if it is both left and right dilation
invariant.
\end{enumerate}
\end{defn}

\noindent Recall that in $FP\left(\langle x_{n}\rangle_{n=1}^{\infty}\right)$
the products are taken in increasing order of indices and the following
definition is taken from \cite[Definition 16.36]{key-12}:
\begin{defn}
Let $\left(S,\cdot\right)$ be a semigroup, let $\langle x_{n}\rangle_{n=1}^{\infty}$
be a sequence in $S$, and let $k\in\mathbb{N}$. Then $AP\left(\langle x_{n}\rangle_{n=1}^{k}\right)$
is the set of all products of terms of $\langle x_{n}\rangle_{n=1}^{k}$in
any order with no repetitions. Similarly $AP\langle x_{n}\rangle_{n=1}^{\infty}$is
the set of all products of terms of $\langle x_{n}\rangle_{n=1}^{\infty}$
in any order with no repetitions.
\end{defn}

For example, for $k=3$, we obtain the following:

\[
AP\left(\langle x_{n}\rangle_{n=1}^{3}\right)=\left\lbrace x_{1},x_{2},x_{3},x_{1}x_{2},x_{1}x_{2},x_{1}x_{3},x_{2}x_{3},x_{2}x_{1},x_{3}x_{2},\right.
\]
\[
\left. \qquad \qquad \qquad \qquad  x_{1}x_{2}x_{3},x_{1}x_{3}x_{2},x_{2}x_{1}x_{3},x_{2}x_{1}x_{3},x_{2}x_{3}x_{1},x_{3}x_{1}x_{2},x_{3}x_{2}x_{1}\right\rbrace 
\]

From \cite[Theorem 16.38]{key-12}, we get the following theorem for
IP$^{\star}$-sets which is our main aim in this section to prove
analogous result for essential $\mathcal{F}^{\star}$-sets.
\begin{thm}
Let $\left(S,+,\cdot\right)$ be a weak ring, let $A$ be an IP$^{\star}$
set in $\left(S,+\right)$, and let $\langle x_{n}\rangle_{n=1}^{\infty}$
be any sequence in $S$. Then there exists a sum subsystem $\langle y_{n}\rangle_{n=1}^{\infty}$
of $\langle x_{n}\rangle_{n=1}^{\infty}$ in S such that \textup{$FS(\langle y_{n}\rangle_{n=1}^{\infty})\cup FP(\langle y_{n}\rangle_{n=1}^{\infty})\subseteq A$.}
\end{thm}

The above theorem is true for any sequence in $S$. But we show that
the above result is true for essential $\mathcal{F}^{\star}$-set
for $\mathcal{F}$-minimal sequences.
\begin{lem}
Let $S$ be a set, let $A\subseteq S$.
\begin{enumerate}
\item If $\left(S,+,\cdot\right)$ is a left weak ring and $\mathcal{F}$
is left dilation invariant family with Ramsay property and $A$ is
an essential $\mathcal{F}^{\star}$-set in $\left(S,+\right)$, then
$sA$ is an essential $\mathcal{F^{\star}}$-set in $(S,+)$.
\item If $\left(S,+,\cdot\right)$ is a right weak ring and $\mathcal{F}$
is right dilation invariant family with Ramsay property and $A$ is
an essential $\mathcal{F}^{\star}$-set in $\left(S,+\right)$, then
$As$ is an essential $\mathcal{F^{\star}}$-set in $(S,+)$.
\item If $\left(S,+,\cdot\right)$ is a weak ring and $\mathcal{F}$ is
dilation invariant family with Ramsay property and $A$ is an essential
$\mathcal{F}^{\star}$-set in $\left(S,+\right)$, then $sAt$ is
an essential $\mathcal{F^{\star}}$-set in $(S,+)$.
\end{enumerate}
\end{lem}

\begin{proof}
It suffices to establish $(1)$ since then $(2)$ follows from a left-right
switch and $(3)$ follows from $(1)$ and $(2)$. If $A$ be an essential
$\mathcal{F}$-set, then by elementary characterization of essential
$\mathcal{F}$-set, we get a downward directed family of $\mathcal{F}$-sets $\left\langle C_F\right\rangle _{F\in I}$ such that for
each $F\in I$ and each $u\in C_F$, there exists $G\in I$
with $C_G\subseteq-u+C_F$. Now consider the downward directed family  $\left\langle sC_F\right\rangle _F\in I$
of $\mathcal{F}$-sets. Now for each $F\in I$ and each $su\in sC_F$, there exists
$G\in I$ with $sC_G\subseteq-su+sC_F$(using the left
distributive law of left weak ring). This proves that $sA$ is an
essential $\mathcal{F}$-set in $(S,+)$ for any $s\in S$.
\end{proof}
\noindent Now we are in position two prove the following two lemmas:
\begin{lem}
Let $S$ be a set, let $A\subseteq S$.
\begin{enumerate}
\item If $\left(S,+,\cdot\right)$ is a left weak ring and $\mathcal{F}$
is left dilation invariant family with Ramsay property such that $\beta\left(\mathcal{F}\right)$ is a subsemigroup
of $\beta S$ and $A$ is
an essential $\mathcal{F}^{\star}$-set in $\left(S,+\right)$, then
$s^{-1}A$ is an essential $\mathcal{F^{\star}}$-set in $(S,+)$.
\item If $\left(S,+,\cdot\right)$ is a right weak ring and $\mathcal{F}$
is right dilation invariant family with Ramsay property such that $\beta\left(\mathcal{F}\right)$ is a subsemigroup
of $\beta S$ and $A$ is
an essential $\mathcal{F}^{\star}$-set in $\left(S,+\right)$, then
$As^{-1}$ is an essential $\mathcal{F^{\star}}$-set in $(S,+)$.
\item If $\left(S,+,\cdot\right)$ is a weak ring and $\mathcal{F}$ is
dilation invariant family with Ramsay property such that $\beta\left(\mathcal{F}\right)$ is a subsemigroup
of $\beta S$ and $A$ is an essential
$\mathcal{F}^{\star}$-set in $\left(S,+\right)$, then $s^{-1}At^{-1}$
is an essential $\mathcal{F^{\star}}$-set in $(S,+)$.
\end{enumerate}
\end{lem}

\begin{proof}
It suffices to establish $(1)$ since then $(1)$ follows from a left-right
switch and $(3)$ follows from $(1)$ and $(2)$. It is sufficient
to show that for any essential $\mathcal{F}$-set $B$, $B\cap s^{-1}A\neq\emptyset$.
Since $B$ is essential $\mathcal{F}$-set , $sB$ is essential $\mathcal{F}$-set
and $A\cap sB\neq\emptyset$. Choose $u\in A\cap sB$ and $v\in B$
such that $u=sv$. Therefore $v\in s^{-1}A$ so $B\cap s^{-1}A\neq\emptyset$.
\end{proof}

\begin{defn}
A sequence $\langle x_{n}\rangle_{n=1}^{\infty}$ in $(S,\cdot )$
is $\mathcal{F}$-minimal sequence if 
\[
\cap_{m=1}^{\infty}cl\left(FP\left(\langle x_{n}\rangle_{n=m}^{\infty}\right)\right)\cap\beta\left(F\right)\neq\emptyset.
\]
\end{defn}

The following is the characterization of $\mathcal{F}$- minimal sequences for arbitrary semigroup.
\begin{thm}
For a left inverse shift invariant family $\mathcal{F}$ in a semigroup $(S,\cdot )$
with Ramsay property such that $\beta\left(\mathcal{F}\right)$ is a subsemigroup
of $\beta S$, the following conditions are equivalent:
\begin{enumerate}
\item $\langle x_{n}\rangle_{n=1}^{\infty}$ is almost $\mathcal{F}$- minimal
sequence.
\item $FP(\langle x_{n}\rangle_{n=1}^{\infty})\in q$, for some $q\in\beta(\mathcal{F})$.
\item There is an idempotent in $\cap_{m=1}^{\infty}cl(FP(\langle x_{n}\rangle_{n=m}^{\infty}))\cap\beta(\mathcal{F})$.
\end{enumerate}
\end{thm}
\begin{proof}
The proof is same as the proof of Theorem \ref{equiv}, and so we omit the proof.
\end{proof}

\noindent We now show that all $\mathcal{F}^{\star}$-set have a substantial
multiplicative property.
\begin{thm}
Let be a $\left(S,+,\cdot\right)$ be a weak ring. Let $\mathcal{F}$
be a left inverse shift invariant and dilation invariant family with Ramsay property such that $\beta\left(\mathcal{F}\right)$ is a subsemigroup
of $\beta S$. Let $\langle x_{n}\rangle_{n=1}^{\infty}$
be a $\mathcal{F}$-minimal sequence and $A$ be an essential $\mathcal{F^{\star}}$-set
in $(S,+)$. Then there exists a sum subsystem $\langle y_{n}\rangle_{n=1}^{\infty}$
of $\langle x_{n}\rangle_{n=1}^{\infty}$ such that\textup{ $FS(\langle y_{n}\rangle_{n=1}^{\infty})\cup AP(\langle y_{n}\rangle_{n=1}^{\infty})\subseteq A$.}
\end{thm}

\begin{proof}
Since $\langle x_{n}\rangle_{n=1}^{\infty}$ is a $\mathcal{F}$-minimal
sequence in $(S,+)$, we can find some essential idempotent $p\in E\left(\beta\left(\mathcal{\mathcal{F}}\right)\right)$
for which $FS(\langle x_{n}\rangle_{n=m}^{\infty})\in p$ for each $m\in p$. Since $A$
is an essential $\mathcal{F}^{\star}$-set for every $s,t\in S$,
$s^{-1}A,t^{-1}A,s^{-1}At^{-1}\in p$. Let $A^{\star}=\{s\in A:-s+A\in p\}$,
then $A^{\star}\in p$. We can choose $y_{1}\in A^{\star}\cap FS(\langle x_{n}\rangle_{n=1}^{\infty})$.
Inductively, let $m\in\mathbb{N}$ and $\langle y_{i}\rangle_{i=1}^{m}$,
$\langle H_{i}\rangle_{i=1}^{m}$ in $\mathcal{P}_{f}(\mathbb{N})$
be chosen with the following properties:
\begin{enumerate}
\item For $i\in\{1,2,\ldots\,,m-1\},\,\text{max}H_{i}<\text{min}H_{i+1}$
\item $\text{If}\,y_{i}={\displaystyle \sum_{t\in H_{i}}x_{t},\,\text{then}\,\sum_{t\in H_{i}}x_{t}\in A^{\star},\,\text{and}\,AP(\langle y_{i}\rangle_{i=1}^{m})\subseteq A^\star}.$
\end{enumerate}
We observe that $\{\,\sum_{t\in H}x_{t}:\,H\in\mathcal{P}_{f}(\mathbb{N}),\,\text{min}H>\text{max}H_{m}\}\in p$.
Let us set $B=\{\,\sum_{t\in H}x_{t}:\,H\in\mathcal{P}_{f}(\mathbb{N}),\,\text{min}H>\text{max}H_{m}\}$,
$E_{1}=FS(\langle y_{t}\rangle_{n=1}^{m})$, and $E_{2}=AP(\langle y_{t}\rangle_{n=1}^{m})$.
Now consider
 $$D=B\cap A^{\star}\cap\bigcap_{s\in E_{1}}\left(-s+A^{\star}\right)\cap\bigcap_{s\in E_{2}}\left(s^{-1}A^{\star}\right)\cap\bigcap_{s\in E_{2}}\left(A^{\star}s^{-1}\right)\cap\bigcap_{s,t\in E_{2}}\left(s^{-1}A^{\star}t^{-1}\right).$$
Then $D\in p$. Choose $y_{m+1}\in D$, and $H_{m+1}\in\mathcal{P}_{f}\left(\mathbb{N}\right)$,
such that $\text{min}H_{m+1}>\text{max}H_{m}$. Putting $y_{m+1}=\sum_{t\in H_{m+1}}x_{t}$,
it shows that the induction can be continued and proves the theorem.
\end{proof}


\vspace{0.1in}
\begin{thebibliography}{10}
{\normalsize{}\bibitem{key-1}C. Adams, N. Hindman, and D. Strauss,
Largeness of the set of finite products in a semigroup, Semigroup
Forum 76 (2008), 276-296.}{\normalsize\par}

{\normalsize{}\bibitem{key-2}V. Bergelson and T. Downarowicz, Large
sets of integers and hierarchy of mixing properties of measure preserving
systems, Colloq. Math. 110 (2008), 117-150.}{\normalsize\par}

{\normalsize{}\bibitem{key-3}V. Bergelson and N. Hindman, On IP|$^{\star}$-sets
and central sets, Combinatorica 14 (1994), 269-277.}{\normalsize\par}

{\normalsize{}\bibitem{key-4}C.Christopherson, Closed ideals in the
Stone-\v{C}ech compactification of a countable semigroup and some
application to ergodic theory and topological dynamics, PhD thesis,
Ohio State University, 2014.}{\normalsize\par}

{\normalsize{}\bibitem[5]{key-5} D. De, Additive and Multiplicative
structure of $C^{\star}$-set, Integers, 14, 2(2014), \#A26.}{\normalsize\par}

{\normalsize{}\bibitem[6]{key-6}D. De, Combined algebraic properties
of central$^{\star}$- sets, Integers 7 (2007), \#A37.}{\normalsize\par}

{\normalsize{}\bibitem[7]{key-7}D. De, N. Hindman, and D. Strauss,
A new and stronger Central Sets Theorem, Fundamenta Mathematicae 199
(2008), 155-175.}{\normalsize\par}

{\normalsize{}\bibitem[8]{key-8}H. Furstenberg, Recurrence in ergodic
theory and combinatorical number theory, Princeton University Press,
Princeton, 1981.}{\normalsize\par}

{\normalsize{}\bibitem[9]{key-9}N. Hindman, Small sets satisfying
the Central Sets Theorem, Combinatorial number theory, 57-63, Walter
de Gruyter, Berlin, 2009.}{\normalsize\par}

{\normalsize{}\bibitem[10]{key-10}N. Hindman, A. Maleki, and D. Strauss,
Central sets and their combinatorial characterization, J. Comb. Theory
(Series A) 74 (1996), 188-208.}{\normalsize\par}

{\normalsize{}\bibitem[11]{key-11}N. Hindman and D. Strauss, A simple
characterization of sets satisfying the Central Sets Theorem, New
York J. Math. 15 (2009), 405-413.}{\normalsize\par}

{\normalsize{}\bibitem[12]{key-12}N. Hindman and D. Strauss, Algebra
in the Stone-\v{C}ech compactication: theory and applications, second
edition, de Gruyter, Berlin, 2012.}{\normalsize\par}
\end{thebibliography}
\end{document}